\documentclass[11pt,twoside]{amsart}
\usepackage[all]{xy}
        \usepackage {amssymb,latexsym,amsthm,amsmath,mathtools}
        \usepackage{enumitem,color}
        
\usepackage{mathtools}

\long\def\symbolfootnote[#1]#2{\begingroup%
\def\thefootnote{\fnsymbol{footnote}}\footnote[#1]{#2}\endgroup}

\newcommand{\tra}{\ensuremath{{}^t}}

\newcommand{\GL}{\textup{GL}}
\newcommand{\SL}{\textup{SL}}
\newcommand{\U}{\textup{U}}

\makeatletter
\def\imod#1{\allowbreak\mkern10mu({\operator@font mod}\,\,#1)}
\makeatother

\newtheorem{theorem}{Theorem}[section]
\newtheorem{lemma}[theorem]{Lemma}
\newtheorem{corollary}[theorem]{Corollary}
\newtheorem{proposition}[theorem]{Proposition}
\newtheorem*{theorem*}{Theorem}
\theoremstyle{definition}

\newtheorem{example}[theorem]{Example}
\numberwithin{equation}{section}

\newcommand{\ignore}[1]{}

\newcommand{\mynote}[1]{}
\begin{document}
\setcounter{section}{0}
\title{Asymptotics of the powers in finite reductive groups}
\author{Amit Kulshrestha}
\address{IISER Mohali, Knowledge City, Sector 81, Mohali 140306, India}
\email{amitk@iisermohali.ac.in}
\author{Rijubrata Kundu}
\address{IISER Pune, Dr. Homi Bhabha Road, Pashan, Pune 411 008, India}
\email{rijubrata8@gmail.com}
\author{Anupam Singh}
\address{IISER Pune, Dr. Homi Bhabha Road, Pashan, Pune 411 008, India}
\email{anupamk18@gmail.com}
\thanks{The first named author acknowledges support of SERB grant EMR/2016/001516. The second named author is recipient of PhD fellowship from NBHM. The third named author is funded by SERB through CRG/2019/000271 for this research.}
\today


\begin{abstract}
Let $G$ be a connected reductive group defined over $\mathbb F_q$. Fix an integer $M\geq 2$, and consider the power map $x\mapsto x^M$ on $G$. We denote the image of $G(\mathbb F_q)$ under this map by $G(\mathbb F_q)^M$ and estimate what proportion of regular semisimple, semisimple and regular elements of $G(\mathbb F_q)$ it contains. We prove that as $q\to\infty$, all of these proportions are equal and provide a formula for the same. We also calculate this more explicitly for the groups $\GL(n,q)$ and $\U(n,q)$.
\end{abstract}

\maketitle

\section{Introduction}
The word maps on finite groups of Lie type and algebraic groups have been studied extensively in the last couple of decades. Larsen, Shalev and Tiep achieved a breakthrough with the solution to Waring problem for finite simple groups and quasi-simple groups (see the excellent survey article by Shalev~\cite{sh} and references therein). Another approach to study group-theoretic problems is to study them statistically (for some highlights of this subject see~\cite{sh1, di}) and get probabilistic results which help understand the asymptotic behaviour. One of the most interesting results of this kind is due to Larsen (see~\cite[Proposition 9]{La}) which states that for any non-trivial word $\omega$ and $\epsilon>0$, there exists $r_0$ such that if $G$ is a finite simple group of Lie type of rank $>r_0$, then $|\omega(G)|> |G|^{1-\epsilon}$.  We would like to study asymptotic estimates of the probability that a regular semisimple, semisimple and regular element is a $M^{th}$ power in finite groups of Lie type. The probability of finding cyclic, regular, regular semisimple elements etc. in the finite classical groups is studied in~\cite{FNP}.  The power map was studied in~\cite{KS} for $\GL(n,q)$ where generating functions for the powers is determined. However, the asymptotic values of these are still not well understood. In~\cite{KS1}, this is done for the group $\SL(2,q)$ where we obtained the asymptotic values. In this article, we study this in a more general setting.

Let $k=\bar {\mathbb F_q}$ and $G$ be a connected reductive group over $k$. Let $F$ be a Frobenius map on $G$ giving rise to a finite group of Lie type $G(\mathbb F_q)=G^F$. Let $M\geq 2$ be a positive integer. We consider the power map $\omega \colon G \rightarrow G$ given by $x\mapsto x^M$. Clearly, this map is defined over $\mathbb F_q$. We consider the image of the set $G(\mathbb F_q)$ under this map, denoted as $G(\mathbb F_q)^M$. Further, we denote the set of $M$-power regular semisimple elements as $G(\mathbb F_q)^M_{rs}=G(\mathbb F_q)^M\cap G(\mathbb F_q)_{rs}$, the set of $M$-power semisimple elements as $G(\mathbb F_q)^M_{ss}=G(\mathbb F_q)^M\cap G(\mathbb F_q)_{ss}$, and $M$-power regular elements as $G(\mathbb F_q)^M_{rg}=G(\mathbb F_q)^M\cap G(\mathbb F_q)_{rg}$.  
We are interested in studying the asymptotic values of the following as $q\to \infty$:
$$\frac{|G(\mathbb F_q)^M|}{|G(\mathbb F_q)|}, \frac{|G(\mathbb F_q)_{rs}^M|}{|G(\mathbb F_q)|}, \frac{|G(\mathbb F_q)_{ss}^M|}{|G(\mathbb F_q)|}, \frac{|G(\mathbb F_q)_{rg}^M|}{|G(\mathbb F_q)|}.$$
We study these quantities when $q\to\infty$ and determine the values. The main theorem is as follows:
\begin{theorem}\label{main-reductive-limit}
Let $G$ be a connected reductive group defined over $\mathbb F_q$ with Frobenius map $F$. Let $M\geq 2$ be an integer. Then, 
\begin{eqnarray*}
\lim_{q\to\infty} \frac{|G(\mathbb F_q)^M|}{|G(\mathbb F_q)|}&=& \lim_{q\to\infty} \frac{|G(\mathbb F_q)^M_{rs}|}{|G(\mathbb F_q)|} = \lim_{q\to\infty} \frac{|G(\mathbb F_q)^M_{ss}|}{|G(\mathbb F_q)|}= \lim_{q\to\infty} \frac{|G(\mathbb F_q)^M_{rg}|}{|G(\mathbb F_q)|}\\
&=& \displaystyle \sum_{T=T_{d_1,\cdots, d_s}} \frac{1}{|W_{T}|(M, d_1)\cdots (M, d_s)} 
\end{eqnarray*}
where the sum varies over non-conjugate maximal tori $T$ in $G(\mathbb F_q)$, $T=T_{d_1,\cdots, d_s}\cong C_{d_1}\times \cdots \times C_{d_s}$ reflects the cyclic structure of $T$, and the group $W_{T}=N_{G(\mathbb F_q)}(T)/T$. 
\end{theorem}
\noindent The proof of this is in Section~\ref{section-RG} and the background and notation are set in Section~\ref{section-GLT}. 

When $M$ is a prime, in Section~\ref{section-GL} and~\ref{section-U}, we obtain the explicit formula for the group $\GL(n,q)$ and the unitary group $\U(n,q)$, respectively. In Section~\ref{section-examples}, we compute some examples by the first principle which serves two purposes: (a) verifies that, indeed, we get the right formula, and (b) helps visualize our main problem. This also brings the limits computed in~\cite{KS1} to the more general context.

\subsection*{Acknowledgement}
We dedicate this paper to Professor B. Sury on the occasion of his $60^{th}$ birthday. He has been a source of encouragement to all of us over several years.

\section{Finite groups of Lie type}\label{section-GLT}

Let $\mathbb F_q$ be a finite field and $k=\bar {\mathbb F_q}$. 
Let $G$ be a connected reductive group over $k$ with Frobenius map $F$, so that $G(\mathbb F_q)=G^F$ is a finite group of Lie type. With this notation, we consider $G(\mathbb F_q) \subset G$. A couple of standard examples are as follows:
\begin{example}
Consider the group $\GL(n)$ over $k$. Define  $F\colon \GL(n, k) \rightarrow \GL(n,k)$ by $(a_{i,j})\mapsto (a_{i,j}^q)$.  This is a Frobenius map and $\GL(n)^F=\GL(n)(\mathbb F_q)=\GL(n,q)$.
\end{example}
\begin{example}
Once again we consider the map $F$ on $\GL(n)$ given by $(a_{i,j}) \mapsto \tra (a_{i,j}^q)^{-1}$. 
The fixed point set is the unitary group $\GL(n)^F=\U(n,q)\subset \GL(n, q^2)$.
\end{example}

We recall some standard facts on the conjugacy classes of maximal tori in $G(\mathbb{F}_q)$. We refer to~\cite{MT} and~\cite{ca} for the details. The maximal tori in $G(\mathbb F_q)$ are obtained from that of $G$ which are $F$-stable. We denote the set of all maximal tori in $G$, which are $F$-stable, by $\tau$. A torus $\bar T\in \tau$ gives a maximal torus $\bar T(\mathbb F_q)$, denoted simply as $T$, of $G(\mathbb F_q)$. Every semisimple element of $G(\mathbb F_q)$ belongs to a maximal torus $T$. Every regular semisimple elements belong to a unique maximal torus. Thus, to understand these elements we need to classify all maximal tori. These are understood up to conjugacy as follows. Let $W$ be the Weyl group of $G$. Then, the $G(\mathbb F_q)$ conjugacy classes of  $F$-stable maximal tori of $G$ are in one-one correspondence with $F$-conjugacy classes in $W$ (see~\cite[Proposition 25.1]{MT}). 

Further, the number of conjugates of a fixed maximal torus can be determined as follows.    
\begin{proposition}\label{number-maximal-tori}
With notation as above, let $T$ be a maximal torus in $G(\mathbb{F}_q)$. Suppose that the conjugacy class of $T$ corresponds to the $F$-conjugacy class of $w$ in $W$. Then, the number of conjugates of $T$ in $G(\mathbb{F}_q)$ is
$$\frac{|G(\mathbb{F}_q)|}{|N_{G(\mathbb{F}_q)}(T)|} = \frac{|G(\mathbb{F}_q)|}{|W_T||T|}=\frac{|G(\mathbb{F}_q)|}{|W^{F}|\cdot|T|}$$
where $W$ is the Weyl group of $G$, and $W_T=N_{G(\mathbb{F}_q)}(T)/T$. Furthermore, $W_T= W^F\cong C_{W,F}(w)=\{ x\in W \mid x^{-1}wF(x)=w \} $.
\end{proposition}
\noindent In the case of $\GL(n,q)$ and $\U(n,q)$, the Frobenius $F$ induces the identity map on the Weyl group $W$, thus the $F$-conjugacy classes of $W$ are simply the conjugacy classes in $W$. Details of this computation can be found in~\cite[Section 2]{gksv}, which we will require later.  Since, a maximal torus $T$ of $G(\mathbb F_q)$ is a finite Abelian group, it can be written as a product of cyclic groups. This point of view will be useful in our study. The cyclic structure of the maximal tori for finite classical groups can be found in~\cite{bg, Za}.
    
For a reductive group $G$, the dimension of maximal tori is called the rank of $G$. We denote it by $r$. We also know that for such $G$, there is a root datum $\Phi$, and we denote $|\Phi^{+}|=N$. Then, $dim(G)=2N+r$ and $|G(\mathbb F_q)|=\mathcal O(q^{2N+r})$.
We would require the following estimate on the regular semisimple elements in $G(\mathbb{F}_q)$ which is~\cite[Lemma 4.5]{jkz}.
\begin{lemma}\label{lemma-jkz}
Let $G$ be a reductive group defined over $\mathbb F_q$. We have the following estimation for regular semisimple elements,
$$|G(\mathbb{F}_q)_{rs}| =  |G(\mathbb{F}_q)|(1+ \mathcal O(q^{-1}))$$
where the constants depend on the type of $G$ only.
\end{lemma}
\noindent This reflects the fact that the regular semisimple elements are dense. In the proof of this Lemma the key ingredient is the following estimate,
$$|x\in T \mid x {\rm\ is\ not\ regular}| = \mathcal O(q^{r-1}).$$
We will use this several times.

\section{Asymptotics of the power map for finite groups of Lie type}\label{section-RG}
    
Let $G$ be a connected reductive group defined over the field $\mathbb F_q$. Let $F$ be a Frobenius map giving rise to a finite group of Lie type denoted as $G(\mathbb F_q) = G^F\subset G$. For an integer $M\geq 2$, define $\omega\colon G \rightarrow G$ by $x\mapsto x^M$. Clearly, this is an algebraic map defined over $\mathbb F_q$. We denote the set $w(G(\mathbb F_q))$ by $G(\mathbb F_q)^M$. In this section, we explore the asymptotics of the ratio $\displaystyle\frac{|G(\mathbb{F}_q)^M|}{|G(\mathbb{F}_q)|}$, as $q\to \infty$. We begin with some preparatory lemma. 
\begin{lemma}\label{lemma-power-abelian}
Let $H$ be a finite Abelian group written as a product of cyclic groups $H=C_{d_1}\times \cdots \times C_{d_s}$. Then, 
$$\frac{|H^M|}{|H|} = \frac{1}{(M,d_1)\cdots (M,d_s)}.$$
\end{lemma}
\begin{proof}
We begin with a cyclic group, i.e, $s=1$ case. Let $H = C_d$ be a finite cyclic group of order $d$. We need to show, $\frac{|C_d^M|}{|C_d|} = \frac{1}{(M,d)}$  where $(M,d)$ denotes the gcd of $M$ and $d$. Consider the map $\omega \colon C_d\rightarrow C_d$ defined by $g\mapsto g^M$. It is a group homomorphism with kernel $ker(\omega)=\{g\in C_d \mid g^M=1\}$. Clearly, elements of the kernel are precisely given by $g^{(M,d)}=1$. Thus, $\frac{|C_d^M|}{|C_d|}=\frac{1}{ker(\omega)}=\frac{1}{(M,d)}$.

Now for an Abelian group $H$, the power map is a group homomorphism. Thus, when $H=C_{d_1}\times\cdots \times  C_{d_s}$, the map $\omega \colon C_{d_1}\times \cdots \times C_{d_s} \rightarrow C_{d_1}\times \cdots \times C_{d_s}$ is $(g_1, \ldots, g_s) \mapsto (g_1^M,\ldots, g_s^M)$. Thus, kernel is given by $(g_1, \ldots ,g_s)$ where $g_i^{(M,d_i)}=1$ for all $i$. This gives the required result.  
\end{proof}
Recall that, a regular semisimple element in $G(\mathbb{F}_q)$ is contained in a unique $F$-stable maximal torus of $G$.  
\begin{lemma}\label{lemma-torus-solution}
Let $\alpha\in G(\mathbb F_q)_{rs}$. Suppose $\alpha$ belongs to the $F$-stable maximal torus $\bar T$. Then, $X^M=\alpha$ has a solution in $G(\mathbb F_q)$ if and only if $Y^M=\alpha$ has a solution in $\bar T(\mathbb F_q)$.  
\end{lemma}
\begin{proof}
Let $A\in G(\mathbb F_q)$ such that $A^M=\alpha$. Write Jordan decomposition $A=A_sA_u$, which implies $A_s^M=\alpha$. Now, every semisimple element belongs to some $F$-stable torus, say $A_s\in \bar T'(\mathbb F_q)$. Then, $\alpha\in \bar T'(\mathbb F_q)$. But, $\alpha$ being regular semisimple, it belongs to a unique maximal torus. Thus, $\bar T'(\mathbb F_q)= \bar T(\mathbb F_q)$, hence the solution $A_s\in \bar T(\mathbb F_q)$.  
\end{proof}

\noindent Now we are ready to get an estimate for $M^{th}$ power regular semisimple elements.

\begin{theorem}\label{main-reductive}
Let $G$ be a reductive group defined over $\mathbb F_q$ with the points below given by the Frobenius map $F$. Then, the proportion of $M^{th}$ power regular semisimple elements in $G(\mathbb F_q)$ is,
$$\frac{|G(\mathbb F_q)^M_{rs}|}{|G(\mathbb F_q)|} = \displaystyle \sum_{T=T_{d_1,\cdots, d_s}} \frac{1}{|W_{T}|(M, d_1)\cdots (M, d_s)} + \mathcal O(q^{-1})$$
where the sum varies over non-conjugate maximal tori $T$ in $G(\mathbb F_q)$, $T=T_{d_1,\cdots, d_s}\cong C_{d_1}\times \cdots \times C_{d_s}$ reflects the cyclic structure, and the group $W_{T}=N_{G(\mathbb F_q)}(T)/T$.
\end{theorem}
\begin{proof}
Since a regular semisimple element of $G(\mathbb F_q)$ belongs to a unique $F$-stable maximal torus, we have,
$$G(\mathbb F_q)^M_{rs} = G(\mathbb F_q)_{rs} \cap G(\mathbb F_q)^M = \bigcup_{\bar T\in \tau} \left( \bar T(\mathbb F_q)_{rs}\cap G(\mathbb F_q)^M\right)$$
where $\tau$ is the set of all $F$-stable maximal tori of $G$.
Now, let $\bar T$ be a $F$-stable maximal torus of $G$ and $T=\bar T(\mathbb F_q)$. Then, from Lemma~\ref{lemma-torus-solution} we have, 
$$\bar T(\mathbb F_q)_{rs}\cap G(\mathbb F_q)^M =  T_{rs}\cap G(\mathbb F_q)^M= T^M\cap G(\mathbb F_q)_{rs} .$$
Suppose the cyclic structure of $T=C_{d_1}\times\cdots\times C_{d_s}$. 
Thus, using the argument in~\cite[Lemma 4.5]{jkz} to prove $T \cap G(\mathbb F_q)_{rs} = q^r + \mathcal O(q^{r-1})$ where it is shown that the non regular elements in $T$ are $\mathcal O(q^{r-1})$, we get,
$$|T^M\cap G(\mathbb F_q)_{rs}| = |T^M| + \mathcal O(q^{r-1}) = \frac{1}{(M,d_1) \cdots (M,d_s)} |T| + \mathcal O(q^{r-1})$$
where $r$ is the dimension of $T$ and the second equality follows from Lemma~\ref{lemma-power-abelian}.
Hence,
\begin{eqnarray*}
\frac{ |G(\mathbb F_q)^M_{rs}|}{|G(\mathbb F_q)|} &=& \frac{1}{|G(\mathbb F_q)|} \sum_{\bar T\in \tau, T=\bar T(\mathbb F_q)} \left(\frac{1}{(M,d_1) \cdots (M,d_s)} |T| + \mathcal O(q^{r-1}) \right)\\
 &=& \left(\sum_{T=T_{d_1, \ldots, d_s}} \frac{1}{(M,d_1) \cdots (M,d_s)}\frac{1 }{|W_T|}\right) + \frac{1}{|W_T||T|}\mathcal O(q^{r-1})
\end{eqnarray*}
where we take $T=T_{d_1, \ldots, d_s}$ up to conjugacy. We note that for a fixed $T$, the number of conjugates is $\frac{|G(\mathbb F_q)| }{|W_T||T|}$ (see Proposition~\ref{number-maximal-tori}). Now, since for any $H$ we have $(q-1)^{dim(H)}\leq |H(\mathbb F_q)|\leq (q+1)^{dim(H)}$ where $dim(H)=2N+r$ and $r$ is rank of $H$, applying this to $T$, we get
$$\frac{ |G(\mathbb F_q)^M_{rs}|}{|G(\mathbb F_q)|} =  \sum_{T=T_{d_1, \ldots, d_s}} \frac{1}{|W_T|(M,d_1) \cdots (M,d_s)} + \mathcal O(q^{-1}).$$
This completes the proof.
\end{proof}
We remark that the quantity $\displaystyle \sum_{T=T_{d_1,\cdots, d_s}} \frac{1}{|W_{T}|(M, d_1)\cdots (M, d_s)}$ is intrinsic to the structure of $G$ with given $M$, even though it seem to involve $q$ (see explicit examples in Section~\ref{example}). 
 \begin{corollary}
With the notation as above we have,
 $$\frac{1}{M^{rank(G)}} \leq \lim_{q\to \infty}\frac{|G(\mathbb F_q)^M_{rs}|}{|G(\mathbb F_q)|} = \sum_{T=T_{d_1,\cdots, d_s}} \frac{1}{|W_{T}|(M, d_1)\cdots (M, d_s)} \leq  1.$$
\end{corollary}
\begin{proof}
The upper end is achieved when $M$ is coprime to the order of all maximal tori (for example, if $M\mid q$), and the lower end is achieved when $M$ divides order of each cyclic factors in all maximal tori. Thus we have,
 $$\frac{1}{M^{rank(G)}} \leq \sum_{T=T_{d_1,\cdots, d_s}} \frac{1}{|W_{T}|M^s} \leq \sum_{T=T_{d_1,\cdots, d_s}} \frac{1}{|W_{T}|(M, d_1)\cdots (M, d_s)}\leq \sum_{T} \frac{1}{|W_{T}|} = 1.$$
\end{proof}
\noindent Note that for a fixed $G$, the limit above depends on varying $M$. One of the interesting questions at this moment is to find out all possible limits for a given group $G$. We take this up in the following sections for $\GL(n,q)$ and $\U(n,q)$.

\begin{lemma}\label{lemma-all-elements}
 Let $G$ be a reductive group defined over $\mathbb F_q$ and $M\geq 2$, an integer. Then, we have 
 \begin{enumerate}
\item $\displaystyle \frac{|G(\mathbb{F}_q)^M|}{|G(\mathbb{F}_q)|} = \frac{|G(\mathbb{F}_q)^M_{rs}|}{|G(\mathbb{F}_q)|} + \mathcal O (q^{-1})$.
\item $\displaystyle\frac{|G(\mathbb{F}_q)^M|}{|G(\mathbb{F}_q)|}  = \frac{|G(\mathbb{F}_q)^M_{ss}|}{|G(\mathbb{F}_q)|} + \mathcal O(q^{-1})$.
\item $\displaystyle\frac{|G(\mathbb{F}_q)^M|}{|G(\mathbb{F}_q)|} = \frac{|G(\mathbb{F}_q)^M_{rg}|}{|G(\mathbb{F}_q)|} + \mathcal O(q^{-1})$.
\end{enumerate}
\end{lemma}
\begin{proof}
To prove (1) we show that $|G(\mathbb{F}_q)^M| = |G(\mathbb{F}_q)^M_{rs}| + \mathcal O (q^{2N+r-1})$. Now,
$$|G(\mathbb F_q)^M| = \\ |G(\mathbb F_q)^M_{rs}| + |G(\mathbb F_q)^M_{nrs}|$$
where nrs refers to non regular semisimple elements, and 
$$|G(\mathbb F_q)^M_{nrs}|\leq |G(\mathbb F_q)_{nrs}| = \mathcal O(q^{2N+r-1})$$
gives us,
$$|G(\mathbb F_q)^M| = \\ |G(\mathbb F_q)^M_{rs}| + \mathcal O(q^{2N+r-1}).$$
Since $|G(\mathbb F_q)|=\mathcal O(q^{2N+r})$ we get the required result.

Now, since 
$$|G(\mathbb{F}_q)^M| + \mathcal O (q^{2N+r-1})= |G(\mathbb F_q)_{rs}^M| \leq |G(\mathbb F_q)_{ss}^M| \leq |G(\mathbb F_q)^M| $$
we get,
$|G(\mathbb{F}_q)_{ss}^M| = |G(\mathbb{F}_q)^M| + \mathcal O (q^{2N+r-1})$.
A similar argument proves the result for regular elements.
\end{proof}

\begin{proof}[{\bf Proof of Theorem~\ref{main-reductive-limit}}]
The proof follows from Lemma~\ref{lemma-all-elements} and Theorem~\ref{main-reductive}. 
\end{proof}

\section{The asymptotic results for powers in $\GL(n,q)$}\label{section-GL}

In this section we want to explore Theorem~\ref{main-reductive-limit} for the group $\GL(n)$ over $\mathbb F_q$. We ask further question as follows: Determine all possible limiting values for a given $M$, that is, what are the possible values of $\displaystyle\sum_{T=T_{d_1,\cdots, d_s}} \frac{1}{|W_{T}|(M, d_1)\cdots (M, d_s)}$ for $\GL(n,q)$. We obtain the group $\GL(n,q)$ from $\GL(n,\bar{\mathbb{F}_q})$, as fixed point set of the usual Frobenius map (see~\cite[Example 21.1]{MT}). The maximal tori for this group is easy to determine (see for example~\cite[Example 25.4]{MT}). We recall the same along with its cyclic structure which we require for our purpose. 

The conjugacy classes of maximal tori in $\GL(n,q)$ are in one-one correspondence with the conjugacy classes of its Weyl group $S_n$. Hence, the non-conjugate maximal tori are parametrized by the partitions of $n$.  We follow the notation for partitions as established in~\cite[Section 2]{KS1}.   For a maximal torus $T$, there exists a partition $\lambda=(n_1, n_2, \ldots, n_s)$ of $n$ such that 
$$T \cong \mathbb F_{q^{n_1}}^* \times \cdots \times \mathbb F_{q^{n_s}}^*.$$
Thus, the cyclic structure is $T \cong C_{q^{n_1}-1}\times \ldots \times C_{q^{n_s}-1}$. Corresponding to a partition $\lambda$ of $n$, let $\sigma_\lambda$ denote the standard element of the conjugacy class of $S_n$ with cycle-type $\lambda$. Let $T$ be a  maximal torus of $\GL(n,q)$ parametrized by the partition $\lambda$ of $ n$. Then, $W_T\cong \mathcal Z_{S_n}(\sigma_{\lambda})$. If we write the partition $\lambda$ in power notation $\lambda=1^{m_1}2^{m_2}\ldots i^{m_i}\ldots$, then $|W_T|= \prod_{i} m_i! i^{m_i}$. Thus, 
\begin{proposition}\label{prop-M-bound}
The proportion of $M^{th}$ powers in $\GL(n,q)$ is as follows,
$$\mathfrak P_{\GL}(n,q ,M):= \frac{|\GL(n,q)^M|}{|\GL(n,q)|} = \sum_{\begin{array}{c} \lambda \vdash n  \\ \lambda=1^{m_1}\ldots i^{m_i}\ldots \end{array}} \prod_i \frac{1}{ m_i! i^{m_i}.(M,q^i-1)^{m_i}} + \mathcal O(q^{-1}).$$
\end{proposition}
\begin{proof}
This is clear from the discussion above, and the Theorem~\ref{main-reductive-limit}. 
\end{proof}

We fix $M$ to be a prime and determine the possible subsequential limits of the set $\mathfrak P_{\GL}(n,q,M)$. 
If $M\mid q$, all semisimple elements of $\GL(n,q)$ (being of order coprime to $q$) remain in $\GL(n,q)^M$. Thus, we get 
$$\lim_{q\to \infty}\mathfrak P_{\GL}(n,q,M)=\lim_{q\to \infty} \frac{|\GL(n,q)^M|}{|\GL(n,q)|} =1.$$  
Now we can assume $M\nmid q$. Denote by $o(q)$ the order of $q$ in $\mathbb Z/M\mathbb Z^{\times}$. For $1\leq a\leq n$, define $\pi_{a}(\lambda)$ to be the number of parts (counted with multiplicity) of $\lambda$ divisible by $a$.
\begin{proposition}
Let $M$ be a prime and $(M,q)=1$.   Then, 
$$\lim_{q\to \infty}\mathfrak P_{\GL}(n,q,M)= \sum_{\begin{array}{c} \lambda \vdash n  \\ \lambda=1^{m_1}\ldots i^{m_i}\ldots \end{array}} \frac{1}{M^{\pi_{o(q)}(\lambda)}} \prod_{i} \frac{1}{m_i! i^{m_i} }.$$
Thus, there are $\nu(M-1)$ subsequential limits of $\mathfrak P_{\GL}(n,q,M)$ as $q\to \infty$, where $\nu(M-1)=|\{a \mid 1\leq a\leq n, a\mid M-1\}|$.
\end{proposition}
\begin{proof}
In view of Proposition~\ref{prop-M-bound}, all we need to find out is when $(M, q^i-1)=M$.
We claim that, $(M, q^i -1)=M$ if and only if $o(q) \mid i$. For if, $M\mid (q^i-1)$, then we have $q^i\equiv 1 \imod M$. Thus, $o(q) \mid i$. This gives the formula.

Now, we know that for the group $\GL(n,q)$, the possible values of $j$ are everything between $1$ to $n$. Combined with the fact that if $M\mid (q^i-1)$ then $M\mid (q^j-1)$ for all $j<i$, the values of $\pi_{o(q)}(\lambda)$ are all possible values of $o(q)$ in $\mathbb Z/M\mathbb Z^{\times}$, which are all factors of $M-1$. 
 \end{proof}
 \noindent The following is immediate:
 \begin{corollary}
Let $M$ be a prime. Then, there are $1+\nu(M-1)$ possible values of  $\displaystyle \lim_{q\to \infty}\mathfrak P_{\GL}(n,q,M)$. The values are $1$ (when $M\mid q$) and 
$$\sum_{\lambda \vdash n}  \frac{1}{M^{\pi_{a}(\lambda)} |\mathcal Z_{S_n}(\sigma_{\lambda})|}$$ 
for every $a \mid (M-1)$ such that $1\leq a\leq n$.
\end{corollary}
\noindent We remark that for the case $M\mid q$, the same formula works where we take $\pi_a(\lambda)=0$ for all $\lambda$. 
\begin{corollary}
For $M=2$, these values are 
$$1 \textup{\ \ and\ \ }  \sum_{\lambda \vdash n}  \frac{1}{2^{\pi(\lambda)} |\mathcal Z_{S_n}(\sigma_{\lambda})|}$$ 
where $\pi(\lambda)$ denotes the number of parts of $\lambda$. 
\end{corollary}
\noindent It is quite easy determine the surjectivity of the power maps for $GL(n,q)$.
\begin{proposition}
Let $M\geq 2$ be a prime and $\omega \colon \GL(n,q) \rightarrow \GL(n,q)$ be the power map $x\mapsto x^M$. Then, $\omega$ is surjective if and only if $(M,q)=1$ and $o(q)>n$ where $o(q)$ is order of $q$ in $(\mathbb Z/M\mathbb Z)^{\times}$. 
\end{proposition}
\begin{proof}
For any finite group $G$, and $M$ a prime, $\omega$ is surjective if and only if $(M,|G|)=1$. Now we know that $|GL(n,q)|=q^{\frac{n(n-1)}{2}}\prod_{i=1}^{n} (q^i-1)$. Hence, the result follows.
\end{proof}

For the group $\SL(n,q)$ the computation is similar with the slight modification due to the structure of maximal tori. The maximal tori, up to conjugation, are given by partitions $\lambda$ of $n$ and the size of $|W_T|=|\mathcal Z_{S_n}(\sigma_{\lambda})|$  as before. Thus,
\begin{proposition}\label{SL-limit}
The asymptotic value of $M^{th}$ powers in $\SL(n,q)$ is,
$$\lim_{q\to\infty} \frac{|\SL(n,q)^M|}{|\SL(n,q)|} = \sum_{\lambda \vdash n} \frac{1}{ (M,q^{\lambda_1}-1)\cdots (M,q^{\lambda_{s-1}}-1) \left(M, \frac{q^{\lambda_s}-1}{q-1}\right)  |\mathcal Z_{S_n}(\sigma_{\lambda})|}$$
where $\lambda=(\lambda_1, \ldots, \lambda_s)$.
 \end{proposition}

\section{The asymptotic results for powers in $\U(n,q)$}\label{section-U}

Similar to Section~\ref{section-GL}, we want to get the estimates in Theorem~\ref{main-reductive-limit}, for the unitary group $\U(n,q)$. Recall that $\U(n,q)$ is obtained from $\GL(n,\bar{\mathbb F_q})$ with the Frobenius map $F\colon (a_{ij})\mapsto \tra (a_{ij}^q)^{-1}$. Thus, $\U(n,q) \leq \GL(n,q^2)$.  Once again, the question is to determine the limit $\sum_{T=T_{d_1,\cdots, d_s}} \frac{1}{|W_{T}|(M, d_1)\cdots (M, d_s)}$ more explicitly. The maximal tori for this group is well known and can be, for example, found in~\cite[Section 2]{gksv}. We recall the same along with its cyclic structure which we require for our purpose. 

Similar to the case of $GL(n,q)$, the conjugacy classes of maximal tori in $\U(n,q)$ are in one-one correspondence with the conjugacy classes of $S_n$. Hence, the non-conjugate maximal tori are parametrized by the partitions of $n$.  For a maximal torus $T$ of $\U(n,q)$, there exists a partition $\lambda=(n_1, n_2, \ldots, n_s)$ of $n$ such that 
$$T \cong  \mathbb M_{n_1}\times \cdots \times \mathbb M_{n_s}$$
where $\mathbb M_r = \{x\in \bar{\mathbb F_q} \mid x^{q^m-(-1)^m}=1\}$.  Thus, the cyclic structure is $T \cong C_{q^{n_1}-(-1)^{n_1}}\times \ldots \times C_{q^{n_s}-(-1)^{n_s}}$. 
Note that when $r$ is even $\mathbb M_r\cong \mathbb F_{q^r}^{*}$, and when $r$ is odd $\mathbb M_r=\{x\in \mathbb F_{q^{2r}} \mid x^{q^r+1}=1\}$. Corresponding to a partition $\lambda$ of $n$, let $\sigma_\lambda$ denote the standard element of the conjugacy class of $S_n$ with cycle-type $\lambda$. Let $T$ be a  maximal torus of $\U(n,q)$ parametrized by the partition $\lambda$ of $ n$. Then, $W_T\cong \mathcal Z_{S_n}(\sigma_{\lambda})$. If we write the partition $\lambda$ in power notation $\lambda=1^{m_1}2^{m_2}\ldots i^{m_i}\ldots$, then $|W_T|= \prod_{i} m_i! i^{m_i}$. Thus, 
\begin{proposition}\label{prop-M-bound}
The proportion of $M^{th}$ powers in $\U(n,q)$ is,
$$\mathfrak P_{\U}(n,q ,M):= \frac{|\U(n,q)^M|}{|\U(n,q)|} = \sum_{\begin{array}{c} \lambda \vdash n  \\ \lambda=1^{m_1}\ldots i^{m_i}\ldots \end{array}} \prod_i \frac{1}{ m_i! i^{m_i}.(M,q^i-(-1)^{i})^{m_i}} + \mathcal O(q^{-1}).$$
\end{proposition}
\begin{proof}
This is clear from the discussion above, and the Theorem~\ref{main-reductive-limit}. 
\end{proof}
\noindent We note that $\displaystyle\lim_{q\to \infty} \mathfrak P_{\U}(n,q ,M) = \lim_{q\to \infty} \mathfrak P_{\GL}(n, -q ,M)$ similar to the Ennola duality.

Let $M$ be a prime. Now, we determine the possible subsequential limits of the set $\mathfrak P_{\U}(n,q,M)$. 
If $M\mid q$, all semisimple elements of $\U(n,q)$ (being of order coprime to $q$) remain in $\U(n,q)^M$. Thus, we get 
$$\lim_{q\to \infty}\mathfrak P_{\U}(n,q,M)=\lim_{q\to \infty} \frac{|\U(n,q)^M|}{|\U(n,q)|} =1.$$  
Now, we can assume $M\nmid q$. Denote by $o(q)$ the order of $q$ in $\mathbb Z/M\mathbb Z^{\times}$. For a partition $\lambda=(n_1, n_2, \ldots, n_s)$, let us denote by $\pi'_{o(q)}(\lambda)$ the number of $n_i$ such that if $n_i$ is even $o(q)\mid n_i$; and if $n_i$ is odd, $o(q)$ is even and $o(q)\mid 2n_i$. 
\begin{proposition}
Let $M> 2$ be a prime and $(M,q)=1$.   Then, 
$$\lim_{q\to \infty}\mathfrak P_{\U}(n,q,M)= \sum_{\begin{array}{c} \lambda \vdash n  \\ \lambda=1^{m_1}\ldots i^{m_i}\ldots \end{array}} \frac{1}{M^{\pi'_{o(q)}(\lambda)}} \prod_{i} \frac{1}{m_i! i^{m_i} } = \sum_{\lambda \vdash n}  \frac{1}{M^{\pi'_{o(q)}(\lambda)} |\mathcal Z_{S_n}(\sigma_{\lambda})|}.$$
\end{proposition}
\begin{proof}
In view of Proposition~\ref{prop-M-bound}, all we need to find out is when $(M, q^i- (-1)^i) = M $. We claim that, $(M, q^i -(-1)^i)=M$ if and only if when $i$ is even $o(q) \mid i$, and when $i$ is odd $o(q)$ is even and $o(q)\mid 2i$. For if $i$ is even, $M\mid (q^i-1)$ if and only if $o(q) \mid i$. If $i$ is odd, $M\mid (q^i+1)$ if and only if $o(q)$ is even and $o(q)\mid 2i$. This gives the formula.
\end{proof}
\begin{proposition}
When $M=2$, and $q$ odd, 
  $$\lim_{q\to \infty}\mathfrak P_{\U}(n,q,2) = \sum_{\lambda \vdash n}  \frac{1}{2^{\pi(\lambda)} |\mathcal Z_{S_n}(\sigma_{\lambda})|} = \lim_{q\to \infty}\mathfrak  P_{\GL}(n,q,2)$$ 
where $\pi(\lambda)$ denotes the number of parts of $\lambda$. 
\end{proposition}
\begin{proof}
This is so because when $q$ is odd, $2\mid (q^i-(-1)^i)$ for all $i$.
\end{proof}

\section{Some Examples}\label{example}\label{section-examples}
In this section we discuss some examples and compute the limits.  For the case of $\GL(n)$, the computations are done with the help from that in~\cite{KS}. Thus, it provides an independent verification of the limits obtained in Section~\ref{section-GL}. 

For the group $\GL(2,q)$, we have two maximal tori up to conjugacy. The, split maximal torus $T_1\cong C_{q-1}\times C_{q-1}$ with $|W_{T_1}|=2$, and the anisotropic torus $T_2\cong C_{q^2-1}$ with $|W_{T_2}| = 2$. Thus from Theorem~\ref{main-reductive-limit}, the probability of finding a $M^{th}$ power in $\GL(2,q)$ is $=\frac{1}{2.(M,q-1)(M,q-1)}  + \frac{1}{2.(M,q^2-1)} $.
\begin{example}
For the group $\GL(2,q)$ and $M=2$ the probability is $\frac{1}{2.(2,q-1)^2}+\frac{1}{2.(2,q^2-1)} =\frac{3}{8}$ if $q$ is odd and $1$ if $q$ is even. Now we present the data following direct computation using that in~\cite{KS}. We have the following:

\begin{center}
\begin{tabular}{ |c|c|c| } \hline
$q$ & $|\GL(2, q)^2|$ & $\displaystyle\lim_{q\to\infty}\frac{|\GL(2, q)^2|}{|\GL(2, q)|}$\\ \hline
odd & $\frac{3}{8}q^4-\frac{5}{8}q^3 +\frac{1}{8}q^2+\frac{5}{8}q-\frac{1}{2}$&  $\frac{3}{8}$ \\  \hline
even & $q^4-2q^3+2q-1$& $1$\\  \hline
\end{tabular}
\end{center}
\vskip2mm
\begin{center}
\begin{tabular}{ |c|c|c|c|} 
\hline
$q$ & $|\GL(2,q)^2_{rg}|$ & $|\GL(2, q)^2_{ss}|$ & $|\GL(2, q)^2_{rs}|$  \\   \hline
odd & $\frac{3}{8}q^4-\frac{5}{8}q^3 +\frac{1}{8}q^2-\frac{3}{8}q-\frac{1}{2}$ & $\frac{3}{8}q^4-\frac{9}{8}q^3 +\frac{5}{8}q^2+\frac{9}{8}q-1$ & $\frac{3}{8}q^4-\frac{9}{8}q^3 +\frac{5}{8}q^2+\frac{1}{8}q$\\   \hline
even & $q^4-2q^3+q$ & $q^4-2q^3+2q-1$ & $q^4-2q^3+q$\\ \hline
\end{tabular}
\end{center}
Hence we get the equalities in Theorem~\ref{main-reductive-limit} for this case.
\end{example}

\begin{example}
For the group $\GL(2,q)$ and $M=3$, the probability as per Theorem~\ref{main-reductive-limit} would be: 
$$
\displaystyle\lim_{q\to\infty}\frac{|\GL(2, q)^3|}{|\GL(2, q)|} =
\begin{cases}
1                   & \textup{\ if\ } q=0\imod 3\\ 
\frac{2}{9} & \textup{\ if\ } q=1\imod 3\\
\frac{2}{3} & \textup{\ if\ } q=2\imod 3.\\
\end{cases}
$$
We have the following:
 
\begin{center}
\begin{tabular}{ |c|c|c| } \hline
$q$ & $|\GL(2, q)^3|$ & $\displaystyle\lim_{q\to\infty}\frac{|\GL(2, q)^3|}{|\GL(2, q)|}$\\ \hline
$0$ & $q^4-2q^3+2q-1$ & $1$ \\ \hline
$1$ & $\frac{2}{9}(q^4-q^3-q^2+q)$ & $\frac{2}{9}$\\  \hline
$2$ & $\frac{2}{3}(q^4-q^3-q^2+q)$ & $\frac{2}{3}$ \\ 	\hline
\end{tabular}
\end{center}
 \vskip2mm
\begin{center}
\begin{tabular}{ |c|c|c|c|} 
\hline
$q$ & $|\GL(2,q)^3_{rg}|$ & $|\GL(2, q)^3_{ss}|$ & $|\GL(2, q)^3_{rs}|$  \\   \hline
0 & $q^4-2q^3+q$ & $q^4-2q^3+2q-1$ & $q^4-2q^3+q$\\ \hline
1 & $\frac{2}{9}q^4-\frac{2}{9}q^3-\frac{2}{9}q^2-\frac{1}{9}q+\frac{1}{3}$ & $\frac{2}{9}q^4-\frac{5}{9}q^3+\frac{1}{9}q^2+\frac{5}{9}q-\frac{1}{3}$ & $\frac{2}{9}q^4-\frac{5}{9}q^3+\frac{1}{9}q^2+\frac{2}{9}$\\ \hline
2 & $\frac{2}{3}q^4-\frac{2}{3}q^3- \frac{2}{3}q^2-\frac{1}{3}q+1$ & $\frac{2}{3}q^4-\frac{5}{3}q^3+ \frac{1}{3}q^2+\frac{5}{3}q-1$ & $\frac{2}{3}q^4-\frac{5}{3}q^3+ \frac{1}{3}q^2+\frac{2}{3}q$. \\\hline
\end{tabular}
\end{center}
\end{example}

Now, for the group $\GL(3,q)$, we have three maximal tori up to conjugacy. The, split maximal torus $T_1\cong C_{q-1}\times C_{q-1}\times C_{q-1}$ with $|W_{T_1}|=6$, the anisotropic torus $T_2\cong C_{q^3-1}$ with $|W_{T_2}| = 3$, and $T_3\cong C_{q^2-1}\times C_{q-1}$ with $|W_{T_3}|=2$. Thus, as per Theorem~\ref{main-reductive-limit}, the probability of finding a $M^{th}$ power in $\GL(3,q)$ is
$$=\frac{1}{6.(M,q-1)^3}  + \frac{1}{3.(M,q^3-1)} +\frac{1}{2.(M,q^2-1)(M,q-1)} .$$
\vskip2mm

\begin{example}
For the group $\GL(3,q)$ and $M=2$ the probability is 
$$=\frac{1}{6.(2,q-1)^3}  + \frac{1}{3.(2,q^3-1)} +\frac{1}{2.(2,q^2-1)(2,q-1)} $$
which is $\frac{5}{16}$ when $q$ is odd, and $1$ if $q$ is even. 
We have the following:

\begin{center}
\begin{tabular}{ |c|c|c| } \hline
$q$ & $|\GL(3, q)^2|$ & $\displaystyle\lim_{q\to\infty}\frac{|\GL(3, q)^2|}{|\GL(3, q)|}$\\ \hline
odd &  $\frac{1}{16}(5q^9-7q^8 -q^7 + 2q^6 + 3q^5 + q^4 - 7q^3 + 4q^2)$ &$\frac{5}{16}$ \\ \hline
even &  $q^9-2q^8+2q^6+q^5-q^4-3q^3+q^2+q$ & $1$  \\   \hline
\end{tabular}
\end{center}
\vskip2mm 
\begin{center}
\begin{tabular}{ |c|c|c|c|} 
\hline
$q$ & $|\GL(3,q)^2_{rg}|$ & $|\GL(3, q)^2_{ss}|$ & $|\GL(3, q)^2_{rs}|$  \\   \hline
odd & $\frac{5}{16}q^9-\frac{7}{16}q^8 -\frac{1}{16}q^7-\frac{3}{8}q^6+\frac{11}{16}q^5$ & $\frac{5}{16}q^9-\frac{11}{16}q^8 +\frac{3}{16}q^7$ & $\frac{5}{16}q^9-\frac{11}{16}q^8 +\frac{3}{16}q^7+\frac{3}{8}q^6$\\
&$+\frac{1}{16}q^4+\frac{9}{16}q^3-\frac{1}{4}q^2-\frac{1}{2}q$ & $+\frac{7}{8}q^6-\frac{5}{16}q^5-\frac{11}{16}q^4 $ & $+\frac{11}{16}q^5-\frac{11}{16}q^4-\frac{3}{16}q^3$\\ 
&& $-\frac{11}{16}q^3+q^2+\frac{1}{2}q-\frac{1}{2}$   &\\ \hline
even & $q^9-2q^8+q^6+2q^5$ & $q^9-2q^8+2q^6-q^4$ &  $q^9-2q^8+q^6+2q^5$\\
& $-q^4-q^3$ & $-2q^3+2q^2+q-1$ & $-q^4-q^3$. \\ \hline
\end{tabular}
\end{center}
\end{example}

\begin{example}
For the group $\GL(3,q)$ and $M=3$ the probability is
$$=\frac{1}{6.(3,q-1)^3}  + \frac{1}{3.(3,q^3-1)} +\frac{1}{2.(3,q^2-1)(3,q-1)} =
\begin{cases}
1 & \textup{\ if\ } q=0\imod 3\\ 
\frac{14}{81} &  \textup{ if\ } q=1\imod 3\\
\frac{2}{3} & \textup{\ if\ } q=2\imod 3.
\end{cases}
$$
We have the following which confirms the above:
 \begin{center}
\begin{tabular}{ |c|c|c| } \hline
$q\imod 3$ & $|\GL(3, q)^3|$ & $\displaystyle\lim_{q\to\infty}\frac{|\GL(3, q)^3|}{|\GL(3, q)|}$\\ \hline
$0$ & $q^9-2q^8+2q^6-q^4-2q^3+2q^2+q-1$ & $1$\\  \hline
$1$ & $\frac{1}{81}(14q^9-14q^8-32q^7+36q^6+14q^5-22q^4+4q^3+54q-54)$ &  $\frac{14}{81}$ \\   \hline
$2$ & $\frac{2}{3}(q^9-q^8-q^7+q^5+q^4-q^3)$ & $\frac{2}{3}$ \\  \hline
\end{tabular}
\end{center}
 \vskip2mm 
 \begin{center}
\begin{tabular}{ |c|c|c|c|} 
\hline
$q$ & $|\GL(3,q)^3_{rg}|$ & $|\GL(3, q)^3_{ss}|$ & $|\GL(3, q)^3_{rs}|$  \\   \hline
$0$ & $q^9 - 2q^8 + q^6 + 2q^5$ & $q^9-2q^8+2q^6-q^4$ & $q^9 - 2q^8 + q^6 + 2q^5$\\  &$- q^4 - q^3$ & $-2q^3+2q^2+q-1$ & $- q^4 - q^3$\\ \hline
$1$ & $\frac{14}{81}q^9-\frac{14}{81}q^8-\frac{32}{81}q^7+\frac{1}{3}q^6+\frac{23}{81}q^5$ & $\frac{14}{81}q^9-\frac{23}{81}q^8-\frac{23}{81}q^7 $ &  $\frac{14}{81}q^9-\frac{23}{81}q^8-\frac{23}{81}q^7+\frac{7}{9}q^6$\\
& $-\frac{22}{81}q^4+\frac{40}{81}q^3-\frac{1}{9}q^2-\frac{1}{3}q$ & $+\frac{8}{9}q^6-\frac{13}{81}q^5 -\frac{58}{81}q^4 $ & $+\frac{23}{81}q^5-\frac{58}{81}q^4+\frac{4}{81}q^3$\\ 
&& $-\frac{5}{81}q^3+\frac{4}{9}q^2+q-1$& \\ \hline
$2$ & $\frac{2}{3}q^9-\frac{2}{3}q^8-\frac{2}{3}q^7-q^6 +\frac{5}{3}q^5$ & $\frac{2}{3}q^9-\frac{5}{3}q^8+\frac{1}{3}q^7+2q^6- \frac{1}{3}q^5$ & $\frac{2}{3}q^9-\frac{5}{3}q^8+\frac{1}{3}q^7+q^6$\\
& $+\frac{2}{3}q^4+\frac{4}{3}q^3-q^2-q$ & $-\frac{4}{3}q^4 -\frac{5}{3}q^3+2q^2+q-1$ & $+\frac{5}{3}q^5-\frac{4}{3}q^4-\frac{2}{3}q^3$. \\ \hline
\end{tabular}
\end{center}
\end{example}

\begin{example}
For the group $\SL(2,q)$, from Proposition~\ref{SL-limit}, we have 
 $$\lim_{q\to \infty} \frac{|\SL(2,q)|^M}{|\SL(2,q)|}= \frac{1}{2(M,q-1)} + \frac{1}{2(M,q+1)}.$$
When $q$ is odd and $M$ is a prime this takes the values 
 $$
\begin{cases}
\frac{1}{2} & \textup{if\ } M=2\\
\frac{M+1}{2M} & \textup{if\ } M \textup{\ coprime to $q$, and divides order of\ } $\SL(2,q)$\\
1 & \textup{otherwise}.\\
\end{cases}
 $$
This explains the limits obtained in~\cite[Theorem 5.1]{KS}.
 \end{example}
\begin{example}
For the group $\U(3,q)$ the maximal tori are $T_1\cong C_{q+1}^3$, $T_2\cong C_{q+1}\times C_{q^2-1}$ and $T_3\cong C_{q^3+1}$.  Thus, the proportion of $M^{th}$ power is
$$=\frac{1}{6.(M, q+1)^3} + \frac{1}{2(M, q+1)(M, q^2-1)} + \frac{1}{3.(M, q^3+1)}.$$
The values for $M=2$ and $3$ are,
\begin{center}
\begin{tabular}{|c|c|c|c|c|}\hline
$q\imod 2$ & $\displaystyle\lim_{q\to\infty}\frac{|\U(3,q)^2|}{|\U(3,q)|}$ & & $q\imod 3$& $\displaystyle\lim_{q\to\infty} \frac{|\U(3,q)^3|}{|\U(3,q)|}$  \\ \hline 
$0$   & $1$  && $0$& $1$\\ \hline
$1$ &$\frac{5}{16}$&& $1$& $\frac{2}{3}$ \\ \hline
&&&$2$&   $\frac{14}{81}$. \\\hline
\end{tabular}
 \end{center}

\end{example}


\end{document}